\documentclass[a4paper,12pt,draft]{amsart}
\usepackage{amsmath, amssymb, amsthm}
\usepackage{braket}
\usepackage{url}

\theoremstyle{definition}  
\theoremstyle{plain}   

\newtheorem{thm}{Theorem}[section]

\newtheorem{Theorem}[thm]{Theorem}

\newtheorem{Lem}[thm]{Lemma}

\theoremstyle{definition}  

\newtheorem{Rem}[thm]{Remark}

\newtheorem{Example}[thm]{Example}

\theoremstyle{definition}  

\newtheorem{Definition}[thm]{Definition}

\newcommand{\ZZ}{\mathbb{Z}}
\newcommand{\CC}{\mathbb{C}}

\newcommand{\KK}{\mathbb{K}}

\newcommand{\LLL}{\mathcal{L}}

\newcommand{\dd}{\boldsymbol{d}}

\newcommand{\vv}{\boldsymbol{v}}
\newcommand{\ww}{\boldsymbol{w}}
\newcommand{\xx}{\boldsymbol{x}}

\newcommand{\zz}{\boldsymbol{z}}

\newcommand{\zero}{\boldsymbol{0}}

\newcommand{\Ann}{\operatorname{Ann}}

\newcommand{\diag}{\operatorname{diag}}

\makeatletter
\DeclareSymbolFont{symbolsC}{U}{txsyc}{m}{n}
\DeclareMathSymbol{\MYPerp}{\mathrel}{symbolsC}{121}
\makeatother

\allowdisplaybreaks[3] 

\newbox{\MyTrashBox}
 
\begin{document}

\title[The eigenvalues of Hessian matrices]
{The eigenvalues of Hessian matrices of the complete and complete bipartite graphs}
\author{Akiko Yazawa}
\thanks{These research results were aided by the fund of Nagano Prefecture to
promote scientific activity. }
\address[Akiko Yazawa]{
Department of Science and Technology,
	Graduate School of Medicine, Science and Technology,
	Shinshu University,
	Matsumoto, Nagano, 390-8621, Japan 
}
\email{yazawa@math.shinshu-u.ac.jp}
\date{}
\keywords{The Hessian matrix; generating functions for spanning trees; complete graphs; complete bipartite graphs; graphic matroids; Kirchhoff polynomials}
\maketitle

\begin{abstract}
In this paper, 
we consider the Hessian matrices $H_{\Gamma}$ of the complete and complete bipartite graphs, 
and the special value of $\tilde H_{\Gamma}$ at $x_{i}=1$ for all $x_{i}$. 
We compute the eigenvalues of $\tilde H_{\Gamma}$. 
We show that one of them is positive and that the others are negative. 
In other words, the metric with respect to the symmetric matrix $\tilde H_{\Gamma}$ is Lorentzian. 
Hence those Hessian $\det (H_{\Gamma})$ are not identically zero. 
As an application, we show the strong Lefschetz property for the Artinian Gorenstein algebra associated to the graphic matroids of the complete and complete bipartite graphs with at most five vertices. 
\end{abstract}


\section{Introduction}\label{Introduction}

In \cite{MR3566530}, 
Maeno and Numata introduced algebras $Q/J_{M}$ and $A_{M}$ for a matroid $M$
to give an algebraic proof of the Sperner property for the lattice $\LLL(M)$ consisting of flats of $M$. 
The algebra $A_{M}$ is defined to be the quotient algebra of the ring of the differential polynomials by the annihilator of the base generating function of $M$. 
By definition, the algebra $A_{M}$ is an Artinian Gorenstein algebra. 
In \cite{MR3566530}, 
they show that $Q/J_{M}$ has the strong Lefshetz property in the narrow sense  (see Definition \ref{SLP} for the definition, see also \cite{MR3112920}) if and only if the lattice $\LLL(M)$ is modular, 
and  that $Q/J_{M}=A_{M}$ if and only if $\LLL(M)$ is modular. 
In \cite{MR3733101}, 
for a simple matroid $M$ with rank $r$ and the ground set $E$, Huh and Wang introduced a graded algebra $B^{\ast}(M)=\bigoplus_{p=0}^{r}B^{p}(M)$, 
which is isomorphic to $Q/J_{M}$. 
They show that the multiplication map 
\begin{align*}
\times L^{r-2p}: B^{p}(M)\to B^{r-p}(M) 
\end{align*}
is injective for $p\leq \frac{r}{2}$ and $L=\sum_{i\in E}y_{i}$. 

Maeno and Numata conjectured that the algebra $A_{M}$ has the strong Lefschetz
property for an arbitrary matroid $M$ in an extended abstract \cite{MR2985388} of the paper \cite{MR3566530}. 
In this paper, 
we discuss the strong Lefschetz property of the Artinian Gorenstein algebra $A_{M}$ for graphic matroids of the complete and the complete bipartite graphs. 

Let us recall the definition of $A_{M}$ in the case of a graphic matroids. 
Let $\Gamma$ be a connected graph with $N$ edges. 
Define $B_{\Gamma}$ to be the set of spanning trees in $\Gamma$. 
Then the \emph{graphic matroid} $M(\Gamma)$ of $\Gamma$ is a matroid whose ground set is the set $E(\Gamma)$ of edges $\Gamma$
and whose basis set is the set $B_{\Gamma}$ of spanning trees in $\Gamma$.  
We assign the variable $x_{i}$ to each edge $i$ of $\Gamma$. 
For the graph $\Gamma$, 
we define the weighted generating function $F_{\Gamma}$ for spanning trees in $\Gamma$, called \emph{Kirchhoff polynomial} of $\Gamma$, by 
\begin{align*}
F_{{\Gamma}}=\sum_{T \in B_{{\Gamma}}}\prod_{e \in E(T)} x_{e}. 
\end{align*}  
Then $A_{M(\Gamma)}$ is $\KK[x_{1}, \ldots, x_{N}]/\Ann(F_{\Gamma})$ (see Section \ref{The Lefschetz property for an algebra constructed from a graph} for the definition of $\Ann(F_{\Gamma})$).

There is a criterion for the strong Lefschetz property of a graded Artinian Gorenstein algebra (see Theorem \ref{criterion} for the detail). 
Roughly speaking, a graded Artinian Gorenstein algebra $A=\bigoplus_{k=0}^{s}A_{k}$ has the strong Lefshetz property if and only if the determinants of the $k$th Hessian matrices of the algebra do not vanish for all $k$, 
where the $k$th Hessian matrix is obtained from a linear basis for the homogeneous space $A_{k}$ of degree $k$. 

Now we define the matrix $H_{\Gamma}$ by 
\begin{align*}
H_{\Gamma}=\left(\frac{\partial}{\partial x_{i}}\frac{\partial}{\partial x_{j}}F_{\Gamma}\right)_{i,j\in E(\Gamma)}
\end{align*}
for a connected graph $\Gamma$. 
We call $H_{\Gamma}$ the \emph{Hessian matrix of the graph} $\Gamma$ and $\det (H_{\Gamma})$ the \emph{Hessian of the graph} $\Gamma$. 
If the determinant of $H_{\Gamma}$ does not vanish, 
then the set of all variables is a linear basis for the homogeneous space of degree one. 
Hence the Hessian matrix of the graph coincidences with the first Hessian matrix of the algebra $A_{M(\Gamma)}$. 

In this paper, 
we consider the Hessian matrices $H_{\Gamma}$ of the complete and complete bipartite graphs, 
and the special value of $\tilde H_{\Gamma}$ at $x_{i}=1$ for all $x_{i}$. 
We compute the eigenvalues of $\tilde H_{\Gamma}$. 
We show that one of them is positive and that the others are negative. 
In other words, the metric with respect to the symmetric matrix $\tilde H_{\Gamma}$ is Lorentzian. 
Hence those Hessian $\det (H_{\Gamma})$ are not identically zero. 
This implies the strong Lefschetz property of the Artinian Gorenstein algebra corresponding to the graphic matroid of the complete graph and the complete bipartite graph with at most five vertices. 

This article is organized as follows: 
In Section \ref{The eigenvectors and eigenvalues of block matrices}, 
we will calculate the eigenvectors and eigenvalues of some block matrices. 
Then we will compute the Hessians of the complete graph and the complete bipartite graph in Section \ref{Main results}. 
In Section \ref{The Lefschetz property for an algebra constructed from a graph}, we will discuss the strong Lefschetz property of the algebra $A_{M}$ corresponding to a graphic matroid. 



\section{The eigenvectors and eigenvalues of block matrices}\label{The eigenvectors and eigenvalues of block matrices} 

In this section, 
we give the eigenvectors and eigenvalues of some block matrices. 
We consider three kinds of block matrices $C, D$ and $M(A, \lambda, \dd)$ (Theorems \ref{big to small}, \ref{Dcharacteristic} and \ref{blockdeterminant}).

Let $l\in \ZZ$, $\dd=(d_{1}, d_{2}, \ldots, d_{l})$, and  $\delta=d_{1}+d_{2}+\cdots+d_{l}$. 
Let $A^{ij}$ be a $d_{i}\times d_{j}$ matrix. 
We consider the $\delta \times \delta$ matrix $A$ defined by 
\begin{align*}
A=\left(A^{ij}\right)_{1\leq i, j\leq l}
\end{align*}
with the $d_{i}\times d_{j}$ matrix $A^{ij}$.

\begin{Lem}\label{block親玉}
Let $A^{ij}\vv_{j}=\bar a_{ij}\vv_{i}$ and $\bar A=(\bar a_{ij})_{1\leq i, j\leq l}$, 
where $\vv_{i}$ and $\vv_{j}$ are vectors of length $d_{i}$ and $d_{j}$, respectively.   
If $\left(w_{i}\right)_{1\leq i \leq l}\in \CC^{l}$ is an eigenvector of $\bar A$ belonging to the eigenvalue $\lambda$, 
then $\xx=\left(w_{i}\vv_{i}\right)_{1\leq i \leq l}\in \CC^{\delta}$ satisfies $A\xx=\lambda\xx$. 
\end{Lem}

\begin{proof}
Since $\left(w_{i}\right)_{1\leq i \leq l}$ is the eigenvector of $\bar A$ belonging to $\lambda$, 
\begin{align*}
\bar A\left(w_{i}\right)_{1\leq i \leq l}=\lambda\left(w_{i}\right)_{1\leq i \leq l}. 
\end{align*}
The $i$th row of the equation implies that 
\begin{align*}
\begin{pmatrix}
\bar a_{i1}&\bar a_{i2}&\cdots&\bar a_{il}
\end{pmatrix}
\begin{pmatrix}
w_{1} \\
w_{2} \\
\vdots \\
w_{l} \\
\end{pmatrix} 
&=\lambda w_{i}
\end{align*}
for all $i$. 
Since $A^{ij}\vv_{j}=\bar a_{ij}\vv_{i}$, 
\begin{align*}
w_{1}A^{i1}\vv_{1}+w_{2}A^{i2}\vv_{2}+\cdots+w_{l}A^{il}\vv_{l}
&=\bar a_{i1}w_{1}\vv_{i}+\bar a_{i2}w_{2}\vv_{i}+\cdots+\bar a_{il}w_{l}\vv_{i} \\
&=(\bar a_{i1}w_{1}+\bar a_{i2}w_{2}+\cdots+\bar a_{il}w_{l})\vv_{i} \\
&=\left(
\begin{pmatrix}
\bar a_{i1}&\bar a_{i2}&\cdots&\bar a_{il}
\end{pmatrix}
\begin{pmatrix}
w_{1} \\
w_{2} \\
\vdots \\
w_{l} \\
\end{pmatrix} 
\right)\vv_{i} \\
&=\lambda w_{i}\vv_{i}
\end{align*}
for all $i$. 
Hence 
\begin{align*}
A\xx
&=
\begin{pmatrix}
A^{11}&A^{12}&\cdots&A^{1l} \\
A^{21}&A^{22}&\cdots&A^{2l} \\
\vdots&\vdots&\ddots&\vdots \\
A^{l1}&A^{l2}&\cdots&A^{ll} \\
\end{pmatrix}
\begin{pmatrix}
w_{1}\vv_{1} \\
w_{2}\vv_{2} \\
\vdots \\
w_{l}\vv_{l} \\
\end{pmatrix} \\
&=
\begin{pmatrix}
w_{1}A^{11}\vv_{1}+w_{2}A^{12}\vv_{2}+\cdots+w_{l}A^{1l}\vv_{l} \\
\vdots \\
w_{1}A^{l1}\vv_{1}+w_{2}A^{l2}\vv_{2}+\cdots+w_{l}A^{ll}\vv_{l}
\end{pmatrix} \\
&=\lambda
\begin{pmatrix}
w_{1}\vv_{1} \\
w_{2}\vv_{2} \\
\vdots \\
w_{l}\vv_{l} \\
\end{pmatrix} \\
&=\lambda\xx. 
\end{align*}
\end{proof}

\begin{Rem}\label{inducezerovector}
The vector $\xx$ may be equal to the zero vector $\zero_{\delta}$ of size $\delta$.  
Hence the vector $\xx$ may not be an eigenvector of $A$. 
\end{Rem}

Let $C_{n}$ be the square matrix 
\begin{align*}
\begin{pmatrix}
0&1&&& \\
&0&1&& \\
&&\ddots&\ddots& \\
&&&0&1 \\
1&&&&0 \\
\end{pmatrix}
\end{align*}
of size $n$. 
Let $\zeta_{n}$ be the $n$th primitive root of unity, and
\begin{align*}
\zz_{n,k}=
\begin{pmatrix}
1 \\
\zeta_{n}^{k} \\
\zeta_{n}^{2k} \\
\vdots \\
\zeta_{n}^{(n-1)k} \\
\end{pmatrix}. 
\end{align*}
Then $\zz_{n,k}$ is an eigenvector of $C_{n}$ belonging to the eigenvalue $\zeta_{n}^{k}$. 
Note that $\zz_{n,k}$ is also an eigenvector of $(C_{n})^{t}$ belonging to the  eigenvalue $\zeta_{n}^{kt}$. 

Let $I_{n}$ be the identity matrix of size $n$, and 
$J_{mn}$ the all-one matrix of size $m \times n$. 
For an $n\times n$ matrix $X$, $\chi_{X}(t)$ denotes the characteristic polynomial $\det(tI_{n}-X)$ of $X$ in the variable $t$, and $X^{(1)}$ denotes the first row of $X$. 
Note that the product $X^{(1)}\zz_{n,k}$ is a scalar.

\begin{Lem}\label{cyclic eigenvector}
Let $A$ be an $n \times n$ cyclic matrix. 
The vector $\zz_{n, k}$ is an eigenvector of $A$ belonging to the eigenvalue $A^{(1)}\zz_{n, k}$. 
Hence 
\begin{align*}
\chi_{A}(t)&=\prod_{k=0}^{n-1}(t-A^{(1)}\zz_{n, k}),  \\
\det A &=\prod_{k=0}^{n-1}A^{(1)}\zz_{n, k}. 
\end{align*}
\end{Lem}

We obtain the following from Lemma \ref{cyclic eigenvector}

\begin{Lem}\label{dcmdeterminant}
Let $\alpha, \lambda \in \CC$ and $A=\alpha J_{nn}+\lambda I_{n}$. 
The vector $\zz_{n,0}$ 
is an eigenvector of $A$ belonging to the eigenvalue $\lambda+n\alpha$. 
For $1\leq k\leq n-1$,  the vector $\zz_{n,k}$ is an eigenvector of $A$ belonging to the eigenvalue $\lambda$.  
Hence 
\begin{align*}
\chi_{A}(t)&=(t-\lambda)^{n-1}(t-(\lambda+n\alpha)), \\
\det A&=\lambda^{n-1}(\lambda+n\alpha). 
\end{align*}
\end{Lem}

Now we consider the block matrices $C, D$ and $M(A, \lambda, \dd)$. 

First we consider the case where $\dd=(n, \ldots, n)$ and each block is cyclic. 
Let us consider the block matrix $C$ whose blocks are $n\times n$ cyclic matrices. 
Let $C$ be $\left(C^{ij}\right)_{1\leq i,j \leq l}$, $C^{ij}$ an $n\times n$ matrix for each $i, j$, and $c_{ij}^{(k)}$ an eigenvalue of $C^{ij}$ associated with an eigenvector $\zz_{n, k}$. 
For $C$ and $0 \leq k \leq n-1$, 
we define the $l\times l$ matrix $\bar C_{k}$ by 
\begin{align*}
\bar C_{k}=\left(c_{ij}^{(k)}\right)_{1\leq i,j \leq l}. 
\end{align*}

\begin{thm}\label{big to small}
Let $\left(w_{i}\right)_{1\leq i \leq l}\in \CC^{l}$ be an eigenvector of $\bar C_{k}$ belonging to the eigenvalue $\lambda$. 
Then $\left(w_{i}\zz_{n,k}\right)_{1\leq i \leq l}\in \CC^{nl}$ is an eigenvector of $C$ belonging to $\lambda$. 
Hence
\begin{align*}
\chi_{C}(t)&=\prod_{k=0}^{n-1}\chi_{\bar C_{k}}(t), \\
\det C&=\prod_{k=0}^{n-1} \det \bar C_{k}. 
\end{align*}
\end{thm}

\begin{proof}
Since $\zz_{n,k}$ is a nonzero vector for any $k$, 
the vector $\left(w_{i}\zz_{n,k}\right)_{1\leq i \leq l}$ is also a nonzero vector. 
Since $C^{ij}$ is cyclic, 
it follows from Lemma \ref{cyclic eigenvector} that $\bar C_{k}$ satisfies the assumption in Lemma \ref{block親玉}. 
\end{proof}


Next we consider the case where $\dd=(2n, 2n, \ldots, 2n, n)$. 
Let $D$ be the block matrix $D=\left(D^{ij}\right)_{1 \leq i,j \leq l}$. 
We assume that $D^{ij}$ is a $2n\times 2n$ cyclic matrix if $1 \leq i, j \leq l-1$ 
and that $D^{ll}$ is an $n\times n$ cyclic matrix. 
Moreover we assume that 
\begin{align}\label{eq}
D^{il}=
\begin{pmatrix}
X_{i} \\
X_{i}
\end{pmatrix},  &&
D^{lj}=
\begin{pmatrix}
Y_{j}&Y_{j}
\end{pmatrix} 
\end{align}
for $1\leq i, j \leq l-1$, 
where $X_{i}$ and $Y_{j}$ are $n \times n$ cyclic matrices. 

For $0\leq k \leq 2n-1$, 
we define the $l\times l$ matrix $\bar D_{k}=(d_{ij}^{(k)})_{1\leq i,j \leq l}$ as follows: 
If $k$ is even, then we define 
\begin{align*}
d_{ij}^{(k)}=
\begin{cases}
(D^{ij})^{(1)}\zz_{2n,k} & \text{if  $1\leq j \leq l-1$},  \\
(D^{ij})^{(1)}\zz_{n,\frac{k}{2}} & \text{if $j=l$}. 
\end{cases}
\end{align*}
If $k$ is odd, 
then
\begin{align*}
d_{ij}^{(k)}=
\begin{cases}
(D^{ij})^{(1)}\zz_{2n,k} & \text{if  $1\leq j \leq l-1$},  \\
0 & \text{if $j=l$}. 
\end{cases}
\end{align*}

\begin{Lem}\label{big to small2}
Let $0 \leq k \leq 2n-1$ and $k$ be even. 
Let $\left(w_{i}\right)_{1\leq i \leq l}\in \CC^{l}$ be an eigenvector of $\bar D_{k}$ belonging to the eigenvalue $\lambda$. 
Then 
\begin{align*}
\begin{pmatrix}
w_{1}\zz_{2n, k} \\
\vdots \\
w_{l-1}\zz_{2n, k} \\
w_{l}\zz_{n, \frac{k}{2}}
\end{pmatrix}
\in \CC^{2(l-1)n}
\end{align*}
is an eigenvector of $D$ belonging to $\lambda$. 
\end{Lem}

\begin{proof}
Now the matrix $D$ satisfies the equation \eqref{eq}, and we have
\begin{align*}
\zz_{2n, k}=
\begin{pmatrix}
\zz_{n,\frac{k}{2}} \\
\zz_{n,\frac{k}{2}}
\end{pmatrix}. 
\end{align*}
Hence Lemmas \ref{block親玉} and \ref{cyclic eigenvector} imply Lemma \ref{big to small2}. 
\end{proof}

\begin{Lem}\label{big to small3}
Let $0 \leq k \leq 2n-1$ and $k$ be odd. 
Let $\ww=\left(w_{i}\right)_{1 \leq i \leq l}$ be an eigenvector of $\bar D_{k}$ belonging to the eigenvector $\lambda$. 
If the vector $\ww$ is linearly independent of the vector 
\begin{align*}
\begin{pmatrix}
\zero_{l-1} \\
1
\end{pmatrix}, 
\end{align*}
then 
\begin{align*}
\begin{pmatrix}
w_{1}\zz_{2n, k} \\
\vdots \\
w_{l-1}\zz_{2n, k} \\
\zero_{n}
\end{pmatrix}
\in \CC^{2(l-1)n}
\end{align*}
is an eigenvector of $D$ belonging to $\lambda$. 
\end{Lem}

\begin{proof}
It can be also shown by Lemmas \ref{block親玉} and \ref{cyclic eigenvector}.  
\end{proof}

\begin{Rem}\label{remeigenvector}
The vector 
\begin{align*}
\left(
\begin{array}{c}
\zero_{l-1}     \\
1       
\end{array}
\right)
\end{align*}
is an eigenvector of $\bar D_{k}$ belonging to the eigenvalue zero. 
The eigenvector, however,  does not induce an eigenvector of $D$. 
See also Remark \ref{inducezerovector}. 
\end{Rem}

Theorem \ref{Dcharacteristic} follows from Lemmas \ref{big to small2} and \ref{big to small3} 

\begin{thm}\label{Dcharacteristic}
The characteristic polynomial of $D$ is  
\begin{align*}
\chi_{D}(t)&=\left(\prod_{k: \text{even}}\chi_{\bar D_{k}}(t)\right)\left(\prod_{k: \text{odd}}\frac{1}{t}\chi_{\bar D_{k}}(t)\right) \\
&=\frac{1}{t^{n}}\prod_{k=0}^{2n-1}\chi_{\bar D_{k}}(t). 
\end{align*}
\end{thm}

Finally we consider the block matrix $M(A, \lambda, \dd)$. 
For a square matrix $A$ of size $l$, $\dd = (d_{1}, d_{2}, \ldots, d_{l})$, and 
$\lambda = (\lambda_{1}, \lambda_{2}, \ldots, \lambda_{l})$, 
we define 
\begin{align*}
T(A, \dd)&=(a_{ij}J_{d_{i}d_{j}})_{1\leq i,j \leq l}, \\
D(\lambda, \dd)&=
\begin{pmatrix}
\lambda_{1} I_{d_{1}} &                            &           & \zero                   \\
                           & \lambda_{2} I_{d_{2}} &           &                           \\
                           &                            & \ddots &                           \\
\zero                    &                            &           & \lambda_{l} I_{d_{l}} 
\end{pmatrix}. 
\end{align*}
We define the square matrix $M(A, \lambda, \dd)$ of size $d_{1}+\cdots+d_{l}$ by 
\begin{align*}
M(A, \lambda, \dd)=T(A, \dd)+D(\lambda, \dd) . 
\end{align*}
We define also the square matrix $\bar M(A, \lambda, \dd)$ of size $l$ by
\begin{align*}
\bar M(A, \lambda, \dd)=\diag(d_{1}, \ldots, d_{l})A+\diag(\lambda_{1}, \ldots, \lambda_{l}), 
\end{align*}
where $\diag(x_{1}, \ldots, x_{l})$ is the diagonal matrix with entries $x_{1}, \ldots, x_{l}$.

\begin{Lem}\label{Meigenvalues}
Let $\lambda=(\lambda_{1}, \ldots, \lambda_{l})$, $\dd=(d_{1}, d_{2}, \ldots, d_{l})$, and $\delta=d_{1}+d_{2}+\cdots+d_{l}$. 

\begin{enumerate}
\item Let $i \in \Set{1, 2, \ldots, l}$. 
Let $1\leq k \leq d_{i}-1$. 
We suppose that $w_{i}=1$ and that for $j \in \Set{1, 2, \ldots, l}\setminus \Set{i}, w_{j}=0$. 
The vector $\left(w_{j}\zz_{d_{j}, k}\right)_{1\leq j \leq l}\in \CC^{\delta}$ is an eigenvector of $M(A, \lambda, \dd)$ belonging to the eigenvalue $\lambda_{i}$. 
\item If $\left(w_{i}\right)_{1\leq i \leq l}\in \CC^{l}$ is an eigenvector of $\bar M(A, \lambda, \dd)$ belonging to the eigenvalue $\mu$, 
then $\left(w_{i}\zz_{d_{j}, 0}\right)_{1\leq i \leq l}\in \CC^{\delta}$ is an eigenvector of $M(A, \lambda, \dd)$ belonging to the eigenvalue $\mu$. 
\end{enumerate}
\end{Lem}

\begin{proof}
Let $A=(a_{ij})$. 
Since $1\leq k \leq d_{i}-1$, it holds that $\zeta_{d_{i}}^{k} \neq 1$. 
Hence
\begin{align*}
a_{ij}J_{d_{i}d_{j}}\zz_{d_{j}, 0}=\zero_{d_{i}}. 
\end{align*}
Therefore the claim $(1)$ follows from Lemmas \ref{block親玉} and \ref{dcmdeterminant}. 
It follows that 
\begin{align*}
a_{ij}J_{d_{i}d_{j}}\zz_{d_{j},0}=d_{j}a_{ij}\zz_{d_{i}, 0}
\end{align*}
for $i\neq j$, and that 
\begin{align*}
\left(a_{ii}J_{d_{i}d_{i}}+\lambda_{i} I_{d_{i}}\right)\zz_{d_{i}, 0}=(d_{i}a_{ii}+\lambda_{i})\zz_{d_{i}, 0}
\end{align*}
for all $i$. 
Then the claim $(2)$ follows from Lemma \ref{block親玉}. 
\end{proof}

Theorem \ref{blockdeterminant} follows from Lemma \ref{Meigenvalues}.  

\begin{thm}\label{blockdeterminant}
For a matrix $A$ of size $l$, $\lambda=(\lambda_{1}, \ldots, \lambda_{l})$ and $\dd=(d_{1}, d_{2}, \ldots, d_{l})$, we have
\begin{align*}
\chi_{M(A, \lambda, \dd)}(t)&=\chi_{\bar M(A, \lambda, \dd)}(t)\prod_{i=1}^{l}(t-\lambda_{i})^{d_{i}-1}, \\
\det M(A, \lambda, \dd) &= \det \bar M(A, \lambda, \dd)\prod_{i=1}^{l} \lambda_{i}^{d_{i}-1} . 
\end{align*}
\end{thm}

\section{Main results}\label{Main results}

In this section, we will compute the Hessian matrices of the complete graph and the complete bipartite graph, defined in Section \ref{Introduction}. 
We define $\tilde H_{\Gamma}$ to be the special value of $H_{\Gamma}$ at $x_{i}=1$ for all $i$. 
Note that $(i,j)$-entry in $\tilde H_{\Gamma}$ is the number of the spanning trees including edges $i, j$.  
First we compute the number of spanning trees containing the edges. 
Next we compute $\det \tilde H_{\Gamma}$ by theorems in Section \ref{The eigenvectors and eigenvalues of block matrices}. 

Terms of graphs in this section follows mainly \cite{MR1271140}.



\subsection{The Hessian of the complete graph}
Here we compute the Hessians of the complete graphs (Theorem \ref{The Hessian of the complete graph}).

Let $n\geq 3$. 
The $(i,j)$-entry in $\tilde H_{K_{n}}$ is the number of trees including edges $i, j$ in $K_{n}$ with $n$ vertices.  
Moon gave the following formula for the number of trees containing a prescribed set of edges. 

\begin{Theorem}[Moon \cite{MR0231755}]\label{Moon}
Let $F$ be a forest with $k$ connected components. 
The number of the trees with $n$ vertices containing $F$  is
\begin{align*}
n^{k-2}\prod_{i=1}^{k}j_{i}, 
\end{align*}
where $j_{i}$ is the number of vertices of each component of $F$. 
\end{Theorem}

Let $\tilde H_{K_{n}}=(h_{ij})$. 
It follows from Theorem \ref{Moon} that
\begin{align*}
h_{ij}=
\begin{cases}
0 & (i=j), \\
3n^{n-4} & (\# i\cap j=1), \\
4n^{n-4} & (\# i\cap j=0). 
\end{cases}
\end{align*}

We can prove the following. 

\begin{thm}\label{keyprop}
The eigenvalues of $\tilde H_{K_{n}}$ are $-2n^{n-4}$, $-n^{n-3}$ and $2(n-2)n^{n-3}$. 
The dimensions of the eigenspaces associate with $-2n^{n-4}$, $-n^{n-3}$ and $2(n-2)n^{n-3}$ are $\binom{n}{2}-n$, $n-1$ and $1$, respectively. 
\end{thm}

\begin{Rem}
Theorem \ref{keyprop} implies that one of the eigenvalues of $\tilde H_{K_{n}}$ is positive and that the others are negative. 
\end{Rem}

Theorem \ref{keyprop} implies Theorem \ref{The Hessian of the complete graph}. 

\begin{thm}\label{The Hessian of the complete graph}
Let $n\geq 3$. 
Then the determinant of $\tilde H_{K_{n}}$ is 
\begin{align*}
{(-1)}^{\binom{n}{2}-1}{2}^{\binom{n}{2}-n+1}n^{n+\binom{n}{2}(n-4)}(n-2). 
\end{align*}
Hence the Hessian $\det H_{K_{n}}$ does not vanish for $n\geq 3$. 
\end{thm}

Now we prove Theorem \ref{keyprop}. 
Let $n\geq 3$. 
We define $a(e, e')$ by
\begin{align*}
a(e, e')=
\begin{cases}
0 & (e=e'), \\
3 & (\#(e\cap e' )=1), \\
4 & \text{otherwise},  
\end{cases}
\end{align*}
for $e, e' \in E(K_{n})$. 

Let $V(K_{n})=\Set{0,1,\ldots,n-1}$. 
We consider a group action on $V(K_{n})$ as follows:
Let $G$ be the cyclic group generated by $\sigma$ of the order $n$.    
For a vertex of $K_{n}$, define
\begin{align*}
\sigma(i)=i+1 \pmod{n}. 
\end{align*}
The action of $G$ on $V(K_{n})$ induces an action on $E(K_{n})$ by
\begin{align*}
\sigma\Set{i,j}=\Set{\sigma(i),\sigma(j)}
\end{align*}
for $\Set{i, j}\in E(K_{n})$. 
Under this notation, 
first we show Theorem \ref{keyprop} in the case where $n$ is odd. 
And then, we show Theorem \ref{keyprop} in the case where $n$ is even. 

Consider the case where $n$ is odd. 
Let $n=2l+1$. 
For $1\leq i,j \leq l$, we define 
\begin{align*}
C^{ij}&=\left(a(\sigma^{k}e_{i}, \sigma^{k'}e_{j})\right)_{0\leq k, k' \leq n-1}, \\
C&=\left(C^{ij}\right)_{1 \leq i, j \leq l}, 
\end{align*}
where $e_{i}=\Set{0,i}\in E(K_{n})$. 
Note that $C^{ij}$ is an $n \times n$ cyclic matrix by the definition of the action. 
Since
\begin{align*}
E(K_{n})=\bigsqcup_{i=1}^{l}\Set{\sigma^{k}e_{i}|0\leq k \leq n-1}, 
\end{align*}
the matrix $n^{n-4}C$ is $\tilde H_{K_{n}}$. 
Let us calculate eigenvalues of $C$ by Theorem \ref{big to small}. 
For $0\leq k \leq n-1$, let
\begin{align*}
\bar C_{k}&=\left((C^{ij})^{(1)}\zz_{n, k}\right)_{1 \leq i,j \leq l}. 
\end{align*}

First we consider the case where $1 \leq k \leq n-1$. 
\begin{Lem}\label{barCodd}
Let $1 \leq k \leq n-1$. 
Then  
\begin{align*}
\bar C_{k}-(-2)I_{l}=\left(-\xi_{i}\xi'_{j}\right)_{1 \leq i,j \leq l}, 
\end{align*}
where 
\begin{align*}
\xi_{i}&=\sum_{v\in e_{i}}\zeta_{n}^{vk}=1+\zeta_{n}^{ik}, \\
\xi'_{j}&=\sum_{v\in e_{j}}\zeta_{n}^{-vk}=1+\zeta_{n}^{-jk}
\end{align*}
for all $i, j$. 
Moreover the rank of $\bar C_{k}-(-2)I_{l}$ is one. 
\end{Lem}

\begin{proof}
We fix $k$ and compute $(C^{ij})^{(1)}\zz_{n,k}$. 
First we consider the case where $e_{i}\neq e_{j}$. 
The edges $e_{i}$ and $\sigma^{l} e_{j}$ share their vertices if and only if 
$l=0$, $l=i$, $j+l=0$, and $j+l=i$. 
Since $e_{i}\neq e_{j}$, 
there does not exist $l$ such that $e_{i}=\sigma^{l} e_{j}$. 
Hence if
$l\in\Set{0,i,-j,i-j}$, 
then 
\begin{align*}
a(e_{i}, \sigma^{l} e_{j})=3, 
\end{align*} 
and if $l\notin\Set{0,i,-j,i-j}$,  
then
\begin{align*}
a(e_{i}, \sigma^{l} e_{j})=4. 
\end{align*}
Therefore
\begin{align*}
(C^{ij})^{(1)}\zz_{n,k}
&=3(\sum_{l\in\Set{0,i,-j,i-j}}\zeta_{n}^{kl})+4(\sum_{l\notin\Set{0,i,-j,i-j}}\zeta_{n}^{kl}) \\
&=-1-\zeta_{n}^{ki}-\zeta_{n}^{-kj}-\zeta_{n}^{k(i-j)} \\
&=-\xi_{i}\xi'_{j}. 
\end{align*}
Next we consider the case where $e_{i}=e_{j}$. 
The edges $e_{i}$ and $\sigma^{l}e_{i}$ share their vertices if and only if $l=0, l=i$ and $l+i=0$. 
If $l=0$, then 
\begin{align*}
a(e_{i}, e_{i})=0, 
\end{align*}
if $l=i$ or $l=-i$, then 
\begin{align*}
a(e_{i}, \sigma^{l}e_{i})=3. 
\end{align*}
Hence if $l \notin \Set{0, i, -i}$, then 
\begin{align*}
a(e_{i}, \sigma^{l}e_{i})=4. 
\end{align*}
Therefore
\begin{align*}
(C^{ij})^{(1)}\zz_{n,k}
&=0\zeta_{n}^{0}+3(\sum_{l\in\Set{i,-i}}\zeta_{n}^{kl})+4(\sum_{l\notin\Set{0,i,-i}}\zeta_{n}^{kl}) \\
&=-4-\zeta_{n}^{ki}-\zeta_{n}^{-ki} \\
&=-\xi_{i}\xi_{i}'-2. 
\end{align*}

We have 
\begin{align*}
\bar C_{k}-(-2)I_{l}&=
\left(-\xi{i}\xi_{j}'\right)_{1 \leq i,j \leq l} \\
&=-
\begin{pmatrix}
\xi_{1} \\
\vdots \\
\xi_{l} 
\end{pmatrix}
\begin{pmatrix}
\xi_{1}'&\cdots&\xi_{l}'
\end{pmatrix}. 
\end{align*}
Hence the rank of $\bar{C}_{k}-(-2)I_{l}$ is one. 

\end{proof}

%

Next we obtain all eigenvalues of $\bar C_{k}$ for $1\leq k \leq n-1$ by computing the trace of $\bar C_{k}$. 

\begin{Lem}\label{Lem6}
For $1 \leq k \leq n-1$, 
the eigenvalues of $\bar C_{k}$ are $-2$ and $-n$. 
The dimensions of the eigenspaces of $\bar C_{k}$ associate with $-2$ and $-n$ are $l-1$ and  $1$, respectively. 
\end{Lem}

\begin{proof}
The trace of $\bar C_{k}$ is 
\begin{align*}
\sum_{i=1}^{l}-\xi_{i}\xi'_{i}
&=\sum_{i=1}^{l} (-4-\zeta_{n}^{ki}-\zeta_{n}^{-ki}) \\
&= -4l+1 \\
&= -4\frac{n-1}{2}+1 \\
&= -2n+3. 
\end{align*}
Therefore it follows from Lemmas \ref{big to small} and \ref{barCodd} that  
the eigenvalues of $\bar C_{k}$ are $-2$ and $-n$, and 
the dimensions of the eigenspaces of $\bar C_{k}$ associate with $-2$ and $-n$ are $l-1$ and  $1$, respectively. 
\end{proof}

Similarly, we obtain the consequence in the case where $k=0$. 

\begin{Lem}\label{Lem9}
The eigenvalues of $\bar C_{0}$ are $-2$ and $2n(n-2)$. 
The dimensions of the eigenspaces of $\bar C_{0}$ associate with $-2$ and $2n(n-2)$ are $l-1$ and $1$, respectively. 
\end{Lem}


We obtain the following result by Theorem \ref{big to small}, and Lemmas \ref{Lem6} and \ref{Lem9}.  

\begin{Lem}\label{Ceigenvalues}
The eigenvalues of $C$ are $-2$, $-n$ and $2n(n-2)$. 
The dimensions of the eigenspaces of $C$ associate with $-2$, $-n$ and $2n(n-2)$ are $\binom{n}{2}-n$, $n-1$ and $1$, respectively. 
\end{Lem}

Next we consider the case where $n$ is even. 
Let $n=2l$. 
We define the matrix $D^{ij}$ by
\begin{align*}
D^{ij}=
\begin{cases}
\left(a(\sigma^{k}e_{i}, \sigma^{k'}e_{j})\right)_{0 \leq k, k' \leq n-1} & \text{for $1\leq i,j \leq l-1$}, \\
\left(a(\sigma^{k}e_{i}, \sigma^{k'}e_{j})\right)
_{\substack{0 \leq k \leq n-1 \\ 0 \leq k' \leq l}}
& \text{for $1\leq i \leq l-1, j=l$}, \\
\left(a(\sigma^{k}e_{i}, \sigma^{k'}e_{j})\right)
_{\substack{0 \leq k \leq l \\ 0 \leq k' \leq n-1}}
& \text{for $i=l, 1\leq j \leq l-1$}, \\
\left(a(\sigma^{k}e_{i}, \sigma^{k'}e_{j})\right)_{0 \leq k, k' \leq l} & \text{for $i=j=l$}, 
\end{cases} 
\end{align*}
and the block matrix $D$ by $D=\left(D^{ij}\right)_{1 \leq i,j \leq l}$. 
Note that the matrix $n^{n-4}D$ is $\tilde H_{K_{n}}$ 
since
\begin{align*}
E(K_{n})=\left(\bigsqcup_{i=1}^{l-1}\Set{\sigma^{k}e_{i}|0 \leq k \leq n-1}\right)\sqcup\Set{\sigma^{k}e_{l}|0 \leq k \leq l}. 
\end{align*} 
For $0\leq j \leq l-1$, the $(1, k)$ entry and $(1, k+l)$ entry in the $(l, j)$ block satisfy   
\begin{align*}
a(e_{l}, \sigma^{k}e_{j})&=a(\sigma^{l}e_{l}, \sigma^{k}e_{j}) \\
&=a(\sigma^{l}\sigma^{l}e_{l}, \sigma^{l}\sigma^{k}e_{j}) \\
&=a(e_{l}, \sigma^{k+l}e_{j}). 
\end{align*}
For $0\leq i \leq l-1$, the $(i, l)$ block also satisfies the assumption \eqref{eq} since $D$ is a symmetric matrix. 
Hence the matrix $D$ in this section satisfies the assumption of the matrix $D$ in Section \ref{The eigenvectors and eigenvalues of block matrices}.

Let us calculate eigenvalues of $D$ by Lemmas \ref{big to small2} and \ref{big to small3}. 
For $0\leq k \leq 2n-1$, 
we define the $l\times l$ matrix $\bar D_{k}=(d_{ij}^{(k)})_{1\leq i,j \leq l}$ as follows: 
If $k$ is even, then we define 
\begin{align*}
d_{ij}^{(k)}=
\begin{cases}
(D^{ij})^{(1)}\zz_{n,k} & \text{if  $1\leq j \leq l-1$},  \\
(D^{ij})^{(1)}\zz_{l,\frac{k}{2}} & \text{if $j=l$}. 
\end{cases}
\end{align*}
If $k$ is odd, 
then we define
\begin{align*}
d_{ij}^{(k)}=
\begin{cases}
(D^{ij})^{(1)}\zz_{n,k} & \text{if  $1\leq j \leq l-1$},  \\
0 & \text{if $j=l$}. 
\end{cases}
\end{align*}

First we consider the case where $1 \leq k \leq 2n-1$.

\begin{Lem}\label{barDeven}
Let $1 \leq k \leq 2n-1$. 
Then  
\begin{align*}
\bar D_{k}-(-2)I_{l}=
\left(-\xi_{i}\xi'_{j}\right)_{1 \leq i,j \leq l}, 
\end{align*}
where 
\begin{align*}
\xi_{i}&=\sum_{v\in e_{i}}\zeta_{n}^{vk}
\end{align*}
for all $i$, and
\begin{align*}
\xi'_{j}&=
\begin{cases}
\sum_{v\in e_{j}}\zeta_{n}^{-vk} & \text{if $1 \leq j \leq l-1$}, \\
\frac{1}{2}\sum_{v\in e_{j}}\zeta_{n}^{-vk} & \text{if $j=l$}.  
\end{cases}
\end{align*}

Moreover the rank of $\bar D_{k}-(-2)I_{l}$ is one. 
\end{Lem}

\begin{proof}
Let $k$ be even. 
We compute the entries $(D^{ij})^{(1)}\zz_{n, k}$, $(D^{il})^{(1)}\zz_{l, \frac{k}{2}}$ and  $(D^{ll})^{(1)}\zz_{l, \frac{k}{2}}$ in $\bar D_{k}$. 

First we consider the case where $1\leq j \leq l-1$. 
In this case we can show the Lemma in the same way as Lemma \ref{barCodd}. 

Next let us consider the case where $1 \leq i \leq l-1$ and $j=l$.  
We compute $(D^{il})^{(1)}\zz_{l, \frac{k}{2}}$ for $1 \leq i \leq l-1$. 
In this case $j=-j$. 
The edges $e_{i}$ and $\sigma^{s} e_{l}$ share their vertices if and only if 
$s=0$, $s=i$. 
Since $e_{i}\neq e_{l}$, 
there does not exist $s$ such that $e_{i}=\sigma^{s} e_{l}$. 
Hence if
$s\in\Set{0,i}$, 
then 
\begin{align*}
a(e_{i}, \sigma^{s} e_{l})=3, 
\end{align*} 
and if $s\notin\Set{0,i}$,  
then
\begin{align*}
a(e_{i}, \sigma^{s} e_{l})=4. 
\end{align*}
Therefore
\begin{align*}
(D^{il})^{(1)}\zz_{l,k}
&=3(\sum_{s\in\Set{0,i}}\zeta_{l}^{ks})+4(\sum_{s\notin\Set{0,i}}\zeta_{l}^{ks}) \\
&=-1-\zeta_{l}^{ki} \\
&=-\xi_{i}\xi'_{l}. 
\end{align*}

Finally we consider the case where $i=j=l$. 
If $s=0$, then
\begin{align*}
a(e_{l}, \sigma^{s}e_{l})=0. 
\end{align*}
If $s\neq 0$, then $e_{l}$ and $\sigma^{s}e_{l}$ do not share their vertices. 
Hence if $s\neq 0$, then  
\begin{align*}
a(e_{s}, \sigma^{s}e_{l})=4. 
\end{align*}
Therefore 
\begin{align*}
(D^{ll})^{(1)}\zz_{l,k}
&=0\cdot\zeta_{l}^{0}+4(\sum_{s=1}^{l-1}\zeta_{l}^{sk}) \\
&=-4 \\
&=-\xi_{l}\xi'_{l}-2. 
\end{align*}

Let $k$ be odd. 
The $l$th column of $\bar D_{k}$ is the zero vector by definition.  
On the other hand, since $k$ is odd, $\zeta_{n}^{lk}=-1$. 
Hence $\xi'_{l}=0$. 

Similarly to Lemma \ref{barCodd}, we can show that the rank of $\bar D_{k}-(-2)I_{l}$ is one.  
\end{proof}

\begin{Lem}\label{Lem6'}
Let $1 \leq k \leq 2n-1$ and $k$ be even. 
The eigenvalues of $\bar D_{k}$ are $-2$ and $-n$. 
The dimensions of the eigenspaces of $\bar D_{k}$ associate with $-2$ and $-n$ are $l-1$ and  $1$, respectively. 
\end{Lem}

\begin{proof}
The trace of $\bar D_{k}$ is 
\begin{align*}
\sum_{i=1}^{l-1}-\xi_{i}\xi'_{i}-4
&=\sum_{i=1}^{l-1} (-4-\zeta_{n}^{ki}-\zeta_{n}^{-ki})-4 \\
&= -4(l-1)+2-4 \\
&= -4l+2 \\
&= -2n+2. 
\end{align*}
Therefore it follows from Lemmas \ref{big to small2} and \ref{barDeven} that  
the eigenvalues of $\bar D_{k}$ are $-2$ and $-n$, and 
the dimensions of the eigenspaces of $\bar D_{k}$ associate with $-2$ and $-n$ are $l-1$ and  $1$, respectively. 
\end{proof}

\begin{Lem}\label{Lem6''}
Let $1 \leq k \leq 2n-1$ and $k$ be odd. 
The eigenvalues of $\bar D_{k}$ are $-2$, $-n$ and $0$. 
The dimensions of the eigenspaces of $\bar D_{k}$ associate with $-2$, $-n$ and $0$ are $l-2$, $1$ and $1$, respectively. 
\end{Lem}

\begin{proof}
Similarly to Lemma \ref{Lem6'}, 
we compute the trace of $\bar D_{k}$ and apply Lemma \ref{big to small3} and Remark \ref{remeigenvector}. 
\end{proof}

Similarly, we obtain the consequence in the case where $k=0$. 

\begin{Lem}\label{Lem9'}
The eigenvalues of $\bar D_{0}$ are $-2$ and $2n(n-2)$. 
The dimensions of the eigenspaces of $\bar D_{0}$ associate with $-2$ and $2n(n-2)$ are $l-1$ and $1$, respectively. 
\end{Lem}

We obtain the following result by Theorem \ref{Dcharacteristic}, and Lemmas \ref{Lem6'}, \ref{Lem6''} and \ref{Lem9'}.  

\begin{Lem}\label{Deigenvalues}
The eigenvalues of $D$ are $-2$, $-n$ and $2n(n-2)$. 
The dimensions of the eigenspaces of $D$ associate with $-2$, $-n$ and $2n(n-2)$ are $\binom{n}{2}-n$, $n-1$ and $1$, respectively. 
\end{Lem}


On combining Lemma \ref{Ceigenvalues} with Lemma \ref{Deigenvalues}, we obtain Theorem \ref{keyprop}. 

\subsection{The Hessian of the complete bipartite graph}
Here we compute the Hessians of the complete bipartite graphs (Theorem \ref{The Hessian of the complete bipartite graph}).

For a graph $\Gamma$, 
we define the \emph{degree matrix} $D_{\Gamma}$ to be 
a diagonal matrix indexed by vertices of $\Gamma$ whose entries are degrees of vertices. 
We also define the \emph{adjacency matrix} $A_{\Gamma}$ to be a matrix indexed by vertices of $\Gamma$ 
whose $(i,j)$-entry is the number of edges with the ends $v_{i}$ and $v_{j}$. 
Note that the entries in $A_{\Gamma}$ are one or zero for a simple graph $\Gamma$. 
We define  the \emph{Laplacian matrix} $L_{\Gamma}$ by $L_{\Gamma}=D_{\Gamma}-A_{\Gamma}$.  
For an arbitrary graph, 
we have the theorem which is the number of spanning trees in the graph. 
See \cite[Theorem 6.3]{MR1271140} for the detail.  

\begin{Theorem}[The matrix-tree theorem]\label{MTT}
Every cofactor of $L_{\Gamma}$ is equal to the number of spanning trees in $\Gamma$. 
\end{Theorem}

For a graph $\Gamma$ and an edge $e$ with ends $v, v'$ of $\Gamma$, 
we define the \emph{contraction} $\Gamma/e$ to be 
the graph obtained by removing the edge $e$ from $\Gamma$ and by putting $v$ in $v'$. 
Let $e, e'$ be edges of $\Gamma$. 
Then $(\Gamma/e)/e'=(\Gamma/e')/e$. 
We write $\Gamma/e,e'$ to denote $(\Gamma/e)/e'$. 

Let $e$ be an edge of a graph $\Gamma$. 
If we apply the matrix-tree theorem to the graph $\Gamma/e$, 
then we obtain the number of spanning trees in $\Gamma$ including the edge $e$.  

Let $X=\Set{0',1',\ldots, (m-1)'}$ and $Y=\Set{0,1,\ldots, n-1}$. 
We consider the complete bipartite graph $K_{m, n}=K_{X, Y}$. 
Theorems \ref{blockdeterminant} and \ref{MTT} implies the following.  

\begin{Lem}\label{the number of spanning trees in the complete bipartite graph including some edges}
Let $h_{ij}$ be the number of spanning trees containing the edges $i, j, i\neq j$ of $K_{m,n}$. 
Then
\begin{align*}
h_{ij}=
\begin{cases}
n^{m-3} m^{n-3} n(2m+n-2) & \text{if $i\cap j\in X$},  \\
n^{m-3} m^{n-3} m(2n+m-2) & \text{if $i\cap j\in Y$},  \\
n^{m-3} m^{n-3} (m+n)(m+n-2) & \text{otherwise}. 
\end{cases}
\end{align*}
\end{Lem}

\begin{proof}
The Laplacian matrix $L_{K_{m,n}}$ is 
\begin{align*}
M
\left(
\begin{pmatrix}
0&-1 \\
-1&0
\end{pmatrix}, 
(n,m), 
(m,n)
\right), 
\end{align*}
where $M$ is as in Section \ref{The eigenvectors and eigenvalues of block matrices}. 
The number of spanning trees including edges $i, j$ is every cofactor of the graph obtained by contraction edges $i, j$ of $K_{m,n}$. 
If $i\cap j \in X$, 
then the Laplacian matrix $L_{K_{m,n}/i, j}$ is 
\begin{align*}
M
\left(
A,
(0,n,m), 
(1,m-1,n-2)
\right),  
\end{align*}
where 
\begin{align*}
A=
\begin{pmatrix}
2m+n-4&-2&-1 \\
-2&0&-1 \\
-1&-1&0
\end{pmatrix}. 
\end{align*}
It follows from Proposition \ref{blockdeterminant} that the $(1,1)$ cofactor of  $L_{K_{m,n}/i,j}$ is 
\begin{align*}
(2m+n-2)n^{m-2}m^{n-3}=m^{n-3}n^{m-3}m(2m+n-2).
\end{align*}
Similarly if $i\cap j \in Y$, then the $(1,1)$ cofactor of $L_{K_{m,n}/i,j}$ is
\begin{align*}
(2m+n-2)n^{m-3}m^{n-2}=n^{m-3}m^{n-3}m(2n+m-2). 
\end{align*}
Finally we consider the case of $i\cap j=\emptyset$. 
Then the Laplacian matrix $L_{K_{m,n}/i,j}$ is 
\begin{align*}
M
\left(
A', 
(0,0,n,m), 
(1,1,m-2,n-2)
\right), 
\end{align*}
where
\begin{align*}
A'=
\begin{pmatrix}
m+n-2&-2&-1&-1 \\
-2&m+n-2&-1&-1 \\
-1&-1&0&-1 \\
-1&-1&-1&0 
\end{pmatrix}. 
\end{align*}
By Proposition \ref{blockdeterminant}, the $(1,1)$ cofactor of $L_{K_{m,n}/i, j}$ is 
\begin{align*}
m^{n-3}n^{m-3}(m+n)(m+n-2). 
\end{align*}
\end{proof}

\begin{Example}
Consider  the complete bipartite graph $K_{2, 3}$, $(m=2, n=3)$.  
The following are the subgraphs of $K_{2, 3}$ consisting four edges including the edges $\Set{0', 0}$ and $\Set{0', 1}$:

\vspace{1ex}
\setlength{\unitlength}{0.8pt}
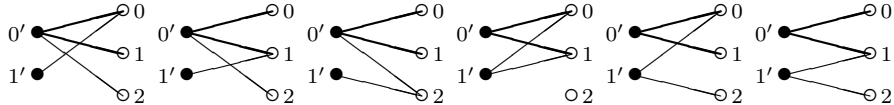
\begin{picture}(65, 50)(-10, 0)
\put(5,15){\makebox(0,0){$\bullet$}}
\put(5,35){\makebox(0,0){$\bullet$}}
\put(45,5){\makebox(0,0){$\circ$}}
\put(45,25){\makebox(0,0){$\circ$}}
\put(45,45){\makebox(0,0){$\circ$}}
\put(0,15){\makebox(0,0)[r]{\tiny$1'$}}
\put(0,35){\makebox(0,0)[r]{\tiny$0'$}}
\put(50,5){\makebox(0,0)[l]{\tiny$2$}}
\put(50,25){\makebox(0,0)[l]{\tiny$1$}}
\put(50,45){\makebox(0,0)[l]{\tiny$0$}}
%
\put(5, 15){\line(4, 3){40}}
\put(5, 35){\line(4, -3){40}}
\thicklines
\put(5, 35){\line(4, -1){40}}
\put(5, 35){\line(4, 1){40}}
\end{picture}
\begin{picture}(65, 50)(-10, 0)
\put(5,15){\makebox(0,0){$\bullet$}}
\put(5,35){\makebox(0,0){$\bullet$}}
\put(45,5){\makebox(0,0){$\circ$}}
\put(45,25){\makebox(0,0){$\circ$}}
\put(45,45){\makebox(0,0){$\circ$}}
\put(0,15){\makebox(0,0)[r]{\tiny$1'$}}
\put(0,35){\makebox(0,0)[r]{\tiny$0'$}}
\put(50,5){\makebox(0,0)[l]{\tiny$2$}}
\put(50,25){\makebox(0,0)[l]{\tiny$1$}}
\put(50,45){\makebox(0,0)[l]{\tiny$0$}}
%
\put(5, 15){\line(4, 1){40}}
%
\put(5, 35){\line(4, -3){40}}
\thicklines
\put(5, 35){\line(4, -1){40}}
\put(5, 35){\line(4, 1){40}}
\end{picture}
\begin{picture}(65, 50)(-10, 0)
\put(5,15){\makebox(0,0){$\bullet$}}
\put(5,35){\makebox(0,0){$\bullet$}}
\put(45,5){\makebox(0,0){$\circ$}}
\put(45,25){\makebox(0,0){$\circ$}}
\put(45,45){\makebox(0,0){$\circ$}}
\put(0,15){\makebox(0,0)[r]{\tiny$1'$}}
\put(0,35){\makebox(0,0)[r]{\tiny$0'$}}
\put(50,5){\makebox(0,0)[l]{\tiny$2$}}
\put(50,25){\makebox(0,0)[l]{\tiny$1$}}
\put(50,45){\makebox(0,0)[l]{\tiny$0$}}
\put(5, 15){\line(4, -1){40}}
%
\put(5, 35){\line(4, -3){40}}
\thicklines
\put(5, 35){\line(4, -1){40}}
\put(5, 35){\line(4, 1){40}}
\end{picture}
\begin{picture}(65, 50)(-10, 0)
\put(5,15){\makebox(0,0){$\bullet$}}
\put(5,35){\makebox(0,0){$\bullet$}}
\put(45,5){\makebox(0,0){$\circ$}}
\put(45,25){\makebox(0,0){$\circ$}}
\put(45,45){\makebox(0,0){$\circ$}}
\put(0,15){\makebox(0,0)[r]{\tiny$1'$}}
\put(0,35){\makebox(0,0)[r]{\tiny$0'$}}
\put(50,5){\makebox(0,0)[l]{\tiny$2$}}
\put(50,25){\makebox(0,0)[l]{\tiny$1$}}
\put(50,45){\makebox(0,0)[l]{\tiny$0$}}
%
\put(5, 15){\line(4, 1){40}}
\put(5, 15){\line(4, 3){40}}
%
\thicklines
\put(5, 35){\line(4, -1){40}}
\put(5, 35){\line(4, 1){40}}
\end{picture}
\begin{picture}(65, 50)(-10, 0)
\put(5,15){\makebox(0,0){$\bullet$}}
\put(5,35){\makebox(0,0){$\bullet$}}
\put(45,5){\makebox(0,0){$\circ$}}
\put(45,25){\makebox(0,0){$\circ$}}
\put(45,45){\makebox(0,0){$\circ$}}
\put(0,15){\makebox(0,0)[r]{\tiny$1'$}}
\put(0,35){\makebox(0,0)[r]{\tiny$0'$}}
\put(50,5){\makebox(0,0)[l]{\tiny$2$}}
\put(50,25){\makebox(0,0)[l]{\tiny$1$}}
\put(50,45){\makebox(0,0)[l]{\tiny$0$}}
\put(5, 15){\line(4, -1){40}}
\put(5, 15){\line(4, 3){40}}
%
\thicklines
\put(5, 35){\line(4, -1){40}}
\put(5, 35){\line(4, 1){40}}
\end{picture}
\begin{picture}(65, 50)(-10, 0)
\put(5,15){\makebox(0,0){$\bullet$}}
\put(5,35){\makebox(0,0){$\bullet$}}
\put(45,5){\makebox(0,0){$\circ$}}
\put(45,25){\makebox(0,0){$\circ$}}
\put(45,45){\makebox(0,0){$\circ$}}
\put(0,15){\makebox(0,0)[r]{\tiny$1'$}}
\put(0,35){\makebox(0,0)[r]{\tiny$0'$}}
\put(50,5){\makebox(0,0)[l]{\tiny$2$}}
\put(50,25){\makebox(0,0)[l]{\tiny$1$}}
\put(50,45){\makebox(0,0)[l]{\tiny$0$}}
\put(5, 15){\line(4, -1){40}}
\put(5, 15){\line(4, 1){40}}
%
\thicklines
\put(5, 35){\line(4, -1){40}}
\put(5, 35){\line(4, 1){40}}
\end{picture}

Hence the number of spanning tree containing the edges $\Set{0', 0}$ and $\Set{0', 1}$ are five.  

The following are the subgraphs of $K_{2, 3}$ consisting four edges including the edges $\Set{0', 0}$ and $\Set{1', 0}$:

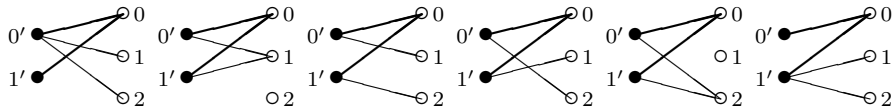
\vspace{1ex}
\begin{picture}(65, 50)(-10, 0)
\put(5,15){\makebox(0,0){$\bullet$}}
\put(5,35){\makebox(0,0){$\bullet$}}
\put(45,5){\makebox(0,0){$\circ$}}
\put(45,25){\makebox(0,0){$\circ$}}
\put(45,45){\makebox(0,0){$\circ$}}
\put(0,15){\makebox(0,0)[r]{\tiny$1'$}}
\put(0,35){\makebox(0,0)[r]{\tiny$0'$}}
\put(50,5){\makebox(0,0)[l]{\tiny$2$}}
\put(50,25){\makebox(0,0)[l]{\tiny$1$}}
\put(50,45){\makebox(0,0)[l]{\tiny$0$}}
%
\thicklines
\put(5, 15){\line(4, 3){40}}
\thinlines
\put(5, 35){\line(4, -3){40}}
\put(5, 35){\line(4, -1){40}}
\thicklines
\put(5, 35){\line(4, 1){40}}
\end{picture}
\begin{picture}(65, 50)(-10, 0)
\put(5,15){\makebox(0,0){$\bullet$}}
\put(5,35){\makebox(0,0){$\bullet$}}
\put(45,5){\makebox(0,0){$\circ$}}
\put(45,25){\makebox(0,0){$\circ$}}
\put(45,45){\makebox(0,0){$\circ$}}
\put(0,15){\makebox(0,0)[r]{\tiny$1'$}}
\put(0,35){\makebox(0,0)[r]{\tiny$0'$}}
\put(50,5){\makebox(0,0)[l]{\tiny$2$}}
\put(50,25){\makebox(0,0)[l]{\tiny$1$}}
\put(50,45){\makebox(0,0)[l]{\tiny$0$}}
%
\put(5, 15){\line(4, 1){40}}
\thicklines
\put(5, 15){\line(4, 3){40}}
\thinlines
\put(5, 35){\line(4, -1){40}}
\thicklines
\put(5, 35){\line(4, 1){40}}
\end{picture}
\begin{picture}(65, 50)(-10, 0)
\put(5,15){\makebox(0,0){$\bullet$}}
\put(5,35){\makebox(0,0){$\bullet$}}
\put(45,5){\makebox(0,0){$\circ$}}
\put(45,25){\makebox(0,0){$\circ$}}
\put(45,45){\makebox(0,0){$\circ$}}
\put(0,15){\makebox(0,0)[r]{\tiny$1'$}}
\put(0,35){\makebox(0,0)[r]{\tiny$0'$}}
\put(50,5){\makebox(0,0)[l]{\tiny$2$}}
\put(50,25){\makebox(0,0)[l]{\tiny$1$}}
\put(50,45){\makebox(0,0)[l]{\tiny$0$}}
\put(5, 15){\line(4, -1){40}}
\thicklines
\put(5, 15){\line(4, 3){40}}
\thinlines
\put(5, 35){\line(4, -1){40}}
\thicklines
\put(5, 35){\line(4, 1){40}}
\end{picture}
\begin{picture}(65, 50)(-10, 0)
\put(5,15){\makebox(0,0){$\bullet$}}
\put(5,35){\makebox(0,0){$\bullet$}}
\put(45,5){\makebox(0,0){$\circ$}}
\put(45,25){\makebox(0,0){$\circ$}}
\put(45,45){\makebox(0,0){$\circ$}}
\put(0,15){\makebox(0,0)[r]{\tiny$1'$}}
\put(0,35){\makebox(0,0)[r]{\tiny$0'$}}
\put(50,5){\makebox(0,0)[l]{\tiny$2$}}
\put(50,25){\makebox(0,0)[l]{\tiny$1$}}
\put(50,45){\makebox(0,0)[l]{\tiny$0$}}
%
\put(5, 15){\line(4, 1){40}}
\thicklines
\put(5, 15){\line(4, 3){40}}
\thinlines
\put(5, 35){\line(4, -3){40}}
\thicklines
\put(5, 35){\line(4, 1){40}}
\end{picture}
\begin{picture}(65, 50)(-10, 0)
\put(5,15){\makebox(0,0){$\bullet$}}
\put(5,35){\makebox(0,0){$\bullet$}}
\put(45,5){\makebox(0,0){$\circ$}}
\put(45,25){\makebox(0,0){$\circ$}}
\put(45,45){\makebox(0,0){$\circ$}}
\put(0,15){\makebox(0,0)[r]{\tiny$1'$}}
\put(0,35){\makebox(0,0)[r]{\tiny$0'$}}
\put(50,5){\makebox(0,0)[l]{\tiny$2$}}
\put(50,25){\makebox(0,0)[l]{\tiny$1$}}
\put(50,45){\makebox(0,0)[l]{\tiny$0$}}
\put(5, 15){\line(4, -1){40}}
\thicklines
\put(5, 15){\line(4, 3){40}}
\thinlines
\put(5, 35){\line(4, -3){40}}
\thicklines
\put(5, 35){\line(4, 1){40}}
\end{picture}
\begin{picture}(65, 50)(-10, 0)
\put(5,15){\makebox(0,0){$\bullet$}}
\put(5,35){\makebox(0,0){$\bullet$}}
\put(45,5){\makebox(0,0){$\circ$}}
\put(45,25){\makebox(0,0){$\circ$}}
\put(45,45){\makebox(0,0){$\circ$}}
\put(0,15){\makebox(0,0)[r]{\tiny$1'$}}
\put(0,35){\makebox(0,0)[r]{\tiny$0'$}}
\put(50,5){\makebox(0,0)[l]{\tiny$2$}}
\put(50,25){\makebox(0,0)[l]{\tiny$1$}}
\put(50,45){\makebox(0,0)[l]{\tiny$0$}}
\put(5, 15){\line(4, -1){40}}
\put(5, 15){\line(4, 1){40}}
\thicklines
\put(5, 15){\line(4, 3){40}}
\thinlines
\thicklines
\put(5, 35){\line(4, 1){40}}
\end{picture}

Hence the number of spanning tree containing the edges $\Set{0', 0}$ and $\Set{1', 0}$ are four.  

The following are the subgraphs of $K_{2, 3}$ consisting four edges including the edges $\Set{0', 0}$ and $\Set{1', 1}$:

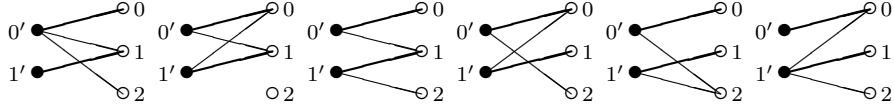
\vspace{1ex}
\begin{picture}(65, 50)(-10, 0)
\put(5,15){\makebox(0,0){$\bullet$}}
\put(5,35){\makebox(0,0){$\bullet$}}
\put(45,5){\makebox(0,0){$\circ$}}
\put(45,25){\makebox(0,0){$\circ$}}
\put(45,45){\makebox(0,0){$\circ$}}
\put(0,15){\makebox(0,0)[r]{\tiny$1'$}}
\put(0,35){\makebox(0,0)[r]{\tiny$0'$}}
\put(50,5){\makebox(0,0)[l]{\tiny$2$}}
\put(50,25){\makebox(0,0)[l]{\tiny$1$}}
\put(50,45){\makebox(0,0)[l]{\tiny$0$}}
%
\thicklines
\put(5, 15){\line(4, 1){40}}
\thinlines
%
\put(5, 35){\line(4, -3){40}}
\put(5, 35){\line(4, -1){40}}
\thicklines
\put(5, 35){\line(4, 1){40}}
\end{picture}
\begin{picture}(65, 50)(-10, 0)
\put(5,15){\makebox(0,0){$\bullet$}}
\put(5,35){\makebox(0,0){$\bullet$}}
\put(45,5){\makebox(0,0){$\circ$}}
\put(45,25){\makebox(0,0){$\circ$}}
\put(45,45){\makebox(0,0){$\circ$}}
\put(0,15){\makebox(0,0)[r]{\tiny$1'$}}
\put(0,35){\makebox(0,0)[r]{\tiny$0'$}}
\put(50,5){\makebox(0,0)[l]{\tiny$2$}}
\put(50,25){\makebox(0,0)[l]{\tiny$1$}}
\put(50,45){\makebox(0,0)[l]{\tiny$0$}}
%
\thicklines
\put(5, 15){\line(4, 1){40}}
\thinlines
\put(5, 15){\line(4, 3){40}}
%
\put(5, 35){\line(4, -1){40}}
\thicklines
\put(5, 35){\line(4, 1){40}}
\end{picture}
\begin{picture}(65, 50)(-10, 0)
\put(5,15){\makebox(0,0){$\bullet$}}
\put(5,35){\makebox(0,0){$\bullet$}}
\put(45,5){\makebox(0,0){$\circ$}}
\put(45,25){\makebox(0,0){$\circ$}}
\put(45,45){\makebox(0,0){$\circ$}}
\put(0,15){\makebox(0,0)[r]{\tiny$1'$}}
\put(0,35){\makebox(0,0)[r]{\tiny$0'$}}
\put(50,5){\makebox(0,0)[l]{\tiny$2$}}
\put(50,25){\makebox(0,0)[l]{\tiny$1$}}
\put(50,45){\makebox(0,0)[l]{\tiny$0$}}
\put(5, 15){\line(4, -1){40}}
\thicklines
\put(5, 15){\line(4, 1){40}}
\thinlines
%
\put(5, 35){\line(4, -1){40}}
\thicklines
\put(5, 35){\line(4, 1){40}}
\end{picture}
\begin{picture}(65, 50)(-10, 0)
\put(5,15){\makebox(0,0){$\bullet$}}
\put(5,35){\makebox(0,0){$\bullet$}}
\put(45,5){\makebox(0,0){$\circ$}}
\put(45,25){\makebox(0,0){$\circ$}}
\put(45,45){\makebox(0,0){$\circ$}}
\put(0,15){\makebox(0,0)[r]{\tiny$1'$}}
\put(0,35){\makebox(0,0)[r]{\tiny$0'$}}
\put(50,5){\makebox(0,0)[l]{\tiny$2$}}
\put(50,25){\makebox(0,0)[l]{\tiny$1$}}
\put(50,45){\makebox(0,0)[l]{\tiny$0$}}
%
\thicklines
\put(5, 15){\line(4, 1){40}}
\thinlines
\put(5, 15){\line(4, 3){40}}
\put(5, 35){\line(4, -3){40}}
\thicklines
\put(5, 35){\line(4, 1){40}}
\end{picture}
\begin{picture}(65, 50)(-10, 0)
\put(5,15){\makebox(0,0){$\bullet$}}
\put(5,35){\makebox(0,0){$\bullet$}}
\put(45,5){\makebox(0,0){$\circ$}}
\put(45,25){\makebox(0,0){$\circ$}}
\put(45,45){\makebox(0,0){$\circ$}}
\put(0,15){\makebox(0,0)[r]{\tiny$1'$}}
\put(0,35){\makebox(0,0)[r]{\tiny$0'$}}
\put(50,5){\makebox(0,0)[l]{\tiny$2$}}
\put(50,25){\makebox(0,0)[l]{\tiny$1$}}
\put(50,45){\makebox(0,0)[l]{\tiny$0$}}
\put(5, 15){\line(4, -1){40}}
\thicklines
\put(5, 15){\line(4, 1){40}}
\thinlines
%
\put(5, 35){\line(4, -3){40}}
\thicklines
\put(5, 35){\line(4, 1){40}}
\end{picture}
\begin{picture}(65, 50)(-10, 0)
\put(5,15){\makebox(0,0){$\bullet$}}
\put(5,35){\makebox(0,0){$\bullet$}}
\put(45,5){\makebox(0,0){$\circ$}}
\put(45,25){\makebox(0,0){$\circ$}}
\put(45,45){\makebox(0,0){$\circ$}}
\put(0,15){\makebox(0,0)[r]{\tiny$1'$}}
\put(0,35){\makebox(0,0)[r]{\tiny$0'$}}
\put(50,5){\makebox(0,0)[l]{\tiny$2$}}
\put(50,25){\makebox(0,0)[l]{\tiny$1$}}
\put(50,45){\makebox(0,0)[l]{\tiny$0$}}
\put(5, 15){\line(4, -1){40}}
\thicklines
\put(5, 15){\line(4, 1){40}}
\thinlines
\put(5, 15){\line(4, 3){40}}
%
\thicklines
\put(5, 35){\line(4, 1){40}}
\end{picture}

Hence the number of spanning tree containing the edges $\Set{0', 0}$ and $\Set{1', 1}$ are five.  
\end{Example}

We can prove the following. 

\begin{thm}\label{K_m,neigenvalues}
Let $m,n\geq 1$ and $m+n\geq 3$. 
Then eigenvalues of $\tilde H_{K_{m,n}}$ are 
$-2m^{n-2}n^{m-2}$, 
$-m^{n-2}n^{m-1}$, 
$-m^{n-1}n^{m-2}$ and
$m^{n-2}n^{m-2}(m+n-1)(m+n-2)$. 
The dimensions of eigenspaces of 
$-2m^{n-2}n^{m-2}$, 
$-m^{n-2}n^{m-1}$, 
$-m^{n-1}n^{m-2}$ and
$m^{n-2}n^{m-2}(m+n-1)(m+n-2)$ are
$(m-1)(n-1)$, 
$n-1$, 
$m-1$ and
$1$, 
respectively. 
\end{thm}

\begin{Rem}
Theorem \ref{K_m,neigenvalues} implies that one of the eigenvalues of $\tilde H_{K_{m,n}}$ is positive and that the others are negative. 
\end{Rem}

%

Theorem \ref{K_m,neigenvalues} implies Theorem \ref{The Hessian of the complete bipartite graph}. 

\begin{thm}\label{The Hessian of the complete bipartite graph}
Let $m,n\geq 1$ and $m+n\geq 3$. 
Then the determinant of $\tilde H_{K_{m,n}}$ is
\begin{multline*}
(-1)^{mn-1}2^{(m-1)(n-1)}m^{(mn-m-1)(n-1)}n^{(mn-n-1)(m-1)} \\
\cdot(m+n-1)(m+n-2). 
\end{multline*}
Hence the Hessian $\det H_{K_{m,n}}$ does not vanish. 
\end{thm}


Let us prove Theorem \ref{K_m,neigenvalues}. 
%
We compute the eigenvalues of the Hessian matrices of $K_{X, Y}=K_{m,n}$. 
For $e, e'\in E(K_{X, Y})$, we define $a(e, e')$ by
\begin{align*}
a(e, e')=
\begin{cases}
0 & (e=e'),  \\
n(2m+n-2) & (e\cap e'\in X), \\
m(2n+m-2) & (e\cap e'\in Y), \\
(m+n)(m+n-2) & \text{otherwise}. 
\end{cases}
\end{align*}
We consider a group action to $V(K_{m,n})$ as follows:
Let $G$ be the cyclic group generated by $\sigma$ of order $n$.    
Let
\begin{align*}
\begin{cases}
\sigma(i')=i' & \text{for $i'\in X=\Set{0', 1', \ldots, (m-1)'}$}, \\
\sigma(i)=i+1 \pmod{n} & \text{for $i \in Y=\Set{0, 1, \ldots, n-1}$}.  
\end{cases}
\end{align*}
The action of $G$ on $V(K_{m,n})$ induces an action on $E(K_{m,n})$ by
\begin{align*}
\sigma\Set{i', j}=\Set{\sigma(i'), \sigma(j)}=\Set{i', \sigma(j)}
\end{align*}
for $i' \in X$ and $j \in Y$. 
We define 
\begin{align*}
C^{ij}&=\left(a(\sigma^{k}e_{i}, \sigma^{k'}e_{j})\right)_{0\leq k, k'\leq n-1}, \\
C&=\left(C^{ij}\right)_{0\leq i. j\leq m-1}, 
\end{align*}
where $e_{i}=\Set{i', 0}\in E(K_{m,n})$. 
Note that $C^{ij}$ is an $n\times n$ cyclic matrix by the definition of the action. 
We have 
\begin{align*}
E(K_{m,n})=\bigsqcup_{i=0}^{m-1}\Set{\sigma^{k}e_{i}|0\leq k\leq n-1}. 
\end{align*}
Hence $m^{n-3}n^{m-3}C$ is $\tilde H_{K_{m,n}}$ by Lemma \ref{the number of spanning trees in the complete bipartite graph including some edges}.

Let us calculate eigenvalues of $C$ by Theorem \ref{big to small}.  
Let 
$a=n(2m+n-2)$, $b=m(2n+m-2)$, $c=(m+n)(m+n-2)$, $A=aJ_{nn}+(-a)I_{n}$, and $B=cJ_{nn}+(b-c)I_{n}$. 
Then we have $C^{ij}=A$ if $i=j$, otherwise $C^{ij}=B$. 
By Lemma \ref{dcmdeterminant}, we obtain the eigenvalues of $A$ and $B$. 
The eigenvalues of $A$ are $(n-1)a$ and $-a$. 
The dimensions of eigenspaces of $(n-1)a$ and $-a$ are $1$ and $n-1$, respectively. 
The eigenvalues of $B$ are $b+(n-1)c$ and $b-c$. 
The dimensions of eigenspaces of $b+(n-1)c$ and $b-c$ are $1$ and $n-1$, respectively. 
For $k=0$, define
\begin{align*}
c_{ij}^{(k)}=
\begin{cases}
(n-1)a & (i=j), \\
b+(n-1)c & (i\neq j). 
\end{cases}
\end{align*}
For $1\leq k\leq m-1$, define
\begin{align*}
c_{ij}^{(k)}=
\begin{cases}
-a & (i=j), \\
b-c & (i\neq j). 
\end{cases}
\end{align*}
For $0\leq k\leq m-1$, we define 
\begin{align*}
\bar C^{(k)}=\left(c_{ij}^{(k)}\right)_{1\leq i,j \leq m}
\end{align*}
Then we have 
\begin{align*}
\bar C^{(0)}&=(b+(n-1)c)J_{mm}+((n-1)a-(b+(n-1)c))I_{m}, \\
\bar C^{(k)}&=(b-c)J_{mm}+(-a-(b-c))I_{m}
\end{align*}
for $1\leq k\leq m-1$. 
By Lemma \ref{dcmdeterminant}, we obtain the eigenvalues of $\bar C^{(0)}$ and $\bar C^{(k)}$. 
The eigenvalues of $\bar C^{(0)}$ are 
\begin{align*}
(n-1)a-(b+(n-1)c)+m(b+(n-1)c) \\
=mn(m+n-1)(m+n-2), 
\end{align*}
and
\begin{align*}
(n-1)a-(b+(n-1)c)=-m^{2} n. 
\end{align*}
The dimensions of the eigenspaces are
$1$ and $m-1$, respectively. 
For $1\leq k\leq m-1$, the eigenvalues of $\bar C^{k}$ are 
\begin{align*}
-a-(b-c)+m(b-c)=-mn^{2}
\end{align*}
and
\begin{align*}
-a-(b-c)=-2mn. 
\end{align*}
The dimensions of the eigenspaces are
$1$ and $m-1$, respectively. 
Theorem \ref{big to small} implies Proposition \ref{K_m,neigenvalues}.



\section{The Lefschetz property for an algebra associated to a graphic matroid}\label{The Lefschetz property for an algebra constructed from a graph}


In this section, we will show the Lefschetz property of the algebra associated to the graphic matroid of the complete graph and the complete bipartite graph with at most five vertices (Theorems \ref{SLP for the complete graph} and \ref{SLP for the complete bipartite graph}).

\begin{Definition}\label{SLP}
Let $A=\bigoplus_{k=0}^{s} A_{k}$, $A_{s}\neq \zero$, 
be a graded Artinian algebra. 
We say that $A$ has the \emph{strong Lefschetz property} 
if there exists an element $L \in A_{1}$ such that the multiplication map $\times L^{s-2k}\colon A_{k}\to A_{s-k}$ is bijective for all $k\leq \frac{s}{2}$. 
We call $L \in A_{1}$ with this property a \emph{strong Lefschetz element}. 
\end{Definition}

Let $\KK$ be a field of characteristic zero.  
For a homogeneous polynomial $F\in \KK[x_{1}, x_{2}, \ldots, x_{N}]$, we define $\Ann(F)$ by
\begin{align*}
\Ann(F)=\Set{P\in \KK[x_{1}, \ldots, x_{N}]| P\left(\frac{\partial}{\partial x_{1}}, 
 \ldots, \frac{\partial}{\partial x_{N}}\right)F=0}. 
\end{align*}
Then $\Ann(F)$ is a homogeneous ideal of $\KK[x_{1}, \ldots, x_{N}]$. 
We consider $A=\KK[x_{1}, \ldots, x_{N}]/\Ann(F)$. 
Since $\Ann(F)$ is homogeneous, the algebra $A$ is graded. 
Furthermore $A$ is an Artinian Gorenstein algebra. 
Conversely, a graded Artinian Gorenstein algebra $A$ has the presentation
\begin{align*}
A=\KK[x_{1}, \ldots, x_{N}]/\Ann(F)
\end{align*}
for some homogeneous polynomial $F\in \KK[x_{1}, x_{2}, \ldots, x_{N}]$. 
We decompose $A$ into the homogeneous components $A_{k}$. 
Let $\Lambda_{k}$ be the basis for $A_{k}$. 
We define the matrix $H_{F}^{(k)}$ by 
\begin{align*}
H_{F}^{(k)}=
\left(
e_{i}\left(\frac{\partial}{\partial x_{1}},\ldots, \frac{\partial}{\partial x_{N}}\right)
e_{j}\left(\frac{\partial}{\partial x_{1}},\ldots, \frac{\partial}{\partial x_{N}}\right)
F
\right)_{e_{i},e_{j}\in \Lambda_{k}}. 
\end{align*}
The determinant of $H_{F}^{(k)}$ is called the $k$th \emph{Hessian} of $F$ with respect to the basis $\Lambda_{k}$. 

\begin{Rem}\label{0th Hessian}
We define the $0$th Hessian of $F$ is $F$.  
\end{Rem}

There is a criterion for the strong Lefschetz property for a graded Artinian Gorenstein algebra.  

\begin{Theorem}[Watanabe \cite{W2000}, Maeno--Watanabe \cite{MR2594646}]\label{criterion}
Consider the graded Artinian Gorenstein algebra $A$ with the following  presentation and decomposition: 
$A=\KK[x_{1}, x_{2}, \ldots, x_{N}]/\Ann(F)=\bigoplus_{k=0}^{s} A_{k}$. 
Let $L=a_{1}x_{1}+a_{2}x_{2}+\cdots+a_{N}x_{N}$. 
The multiplication map 
$\times L^{s-2k}\colon A_{k}\to A_{s-k}$ is bijective 
if and only if 
\begin{align*}
\det H_{F}^{(k)}(a_{1}, a_{2}, \ldots, a_{N})\neq 0. 
\end{align*}
\end{Theorem}

\begin{Definition}
For a graph $\Gamma$ with $N$ edges, we define the graded Artinian Gorenstein algebra $A_{M(\Gamma)}$ by 
\begin{align*}
A_{M(\Gamma)}=\KK[x_{1}, x_{2}, \ldots, x_{N}]/\Ann(F_{\Gamma}). 
\end{align*}
\end{Definition}
If a graph $\Gamma$ has $n+1$ vertices, then the top degree of $A_{M(\Gamma)}$ is $n$.

\begin{Theorem}\label{SLP for the complete graph}
The algebra $A_{M(K_{n})}$ has the strong Lefschetz property for $n \leq5$. 
The element $x_{1}+\cdots+x_{N}$ is a strong Lefschetz element.  
\end{Theorem}
\begin{proof}
Let $n \leq5$. 
It follows from Theorem \ref{The Hessian of the complete graph} that $H_{F_{K_{n}}}^{(1)}=H_{K_{n}}$, and $\det H_{F_{K_{n}}}^{(1)}(1, \ldots, 1)=\det \tilde H_{K_{n}}\neq 0$. 
Hence $A_{K_{n}}$ has the strong Lefschetz property, and 
the element $x_{1}+\cdots+x_{N}$ is a strong Lefschetz element.  
\end{proof}

Similarly Theorem \ref{SLP for the complete bipartite graph} follows from Remark \ref{0th Hessian} and Theorem \ref{The Hessian of the complete bipartite graph}. 

\begin{Theorem}\label{SLP for the complete bipartite graph}
The algebra $A_{M(K_{m,n})}$ has the strong Lefschetz property for $m+n \leq5$ and $m,n\geq 1$. 
The element $x_{1}+\cdots+x_{N}$ is a strong Lefschetz element.  
\end{Theorem}

\begin{Rem}
In \cite{NY2019}, it is shown that the Hessian of any simple graph does not vanish. 
In the paper, it is shown that the Kirchhoff polynomial of a complete graph is the irreducible relative invariant of a prehomogeneous vector space. 
By this fact, we can compute the Hessian as polynomials of the complete graph, and compute the Hessians of any graphs. 
\end{Rem}



\bibliographystyle{amsplain-url}
\bibliography{mr}


\end{document}